\documentclass[12pt,psamsfonts]{article}
\usepackage[dvips]{epsfig}
\usepackage{graphics}
\usepackage{amsmath}
\usepackage{amsfonts}
\usepackage{amssymb}
\usepackage{amsxtra}
\usepackage{eufrak}
\usepackage[mathscr]{eucal}
\usepackage{amsthm}
\usepackage{bm}
\usepackage{enumerate}
\usepackage[mathscr]{eucal}
\usepackage{inputenc}
\usepackage[T2A]{fontenc}
\DeclareSymbolFont{rsfs}{U}{rsfs}{m}{n}
\DeclareSymbolFontAlphabet{\mathscrn}{rsfs} \theoremstyle{plain}
\newtheorem{theorem}{Theorem}[section]
\newtheorem{lemma}{Lemma}[section]
\newtheorem{definition}{Definition}[section]

 \numberwithin{equation}{section}

\begin{document}
\title{\textbf{The truncated Fourier operator. III}}
\author{\textbf{Victor Katsnelson} \and \textbf{Ronny Machluf}}
\footnotetext{\hspace*{-4.0ex}\textbf{Mathematics Subject
Classification: (2000).} Primary 47A38; Secondary 47B35, 47B06,
47A10.\endgraf \hspace*{-2.0ex}\textbf{Keywords:}
Fourier-Plancherel operator, quantum harmonic oscillator, the
ladder operators method, creation and annihilation operators,
 the Melline transform.}
\date{\ }
 \maketitle
 \abstract{The spectral theory of the Fourier operator
 (non-truncated) is expounded. The known construction of
 basis of eigenvectors consisting of the Hermite functions
 is presented. The detail description of the eigenspaces in the spirit
 of a work by Hardy and Titchmarsh is done.}

\setcounter{section}{2}
\section{Spectral theory of the non-truncated Fourier-Plancherel operator}
This manuscript may be considers a continuation of the manuscripts
\cite{KaMa1} and \cite{KaMa2}, but can be read independently of
them.

The  manuscript contains Section 3 of our work. In this section we
expound the spectral theory of the non-truncated Fourier operator,
that is of the operator \(\mathscr{F}_E\), where \(E=\mathbb{R}\).
Here we denote this operator by \(\mathscr{F}\) omitting the
subscript \(E\).

The spectral theory of the operator \(\mathscr{F}\) is closely
related to the spectral theory of the Hermite differential
operator%
 \begin{equation}%
 \mathscr{L}=-\frac{d^2\,\,}{dt^2}+t^2.%
 \end{equation}
The eigenfunction and eigenvalues of the Hermite differential
operator can be found explicitly without directly solving the
differential equation, but using an algebraic method knows as
\emph{method of ladder operators}. This method was found by
P.A.M.\,Dirac. (See \cite[Chapt.\,6,\,Sect.34]{Dir}. The first
edition of this book appeared on 1930.) See also \cite[Chapter 6,
Section 6-2]{DiWi}.
 In the physical literature
the operator which we named as the Hermite operator is usually
named as the Hamiltonian of the one-dimensional quantum harmonic
oscillator. One of the mail algebraic properties of the Fourier
transform \(\mathscr{F}\) on the whole real axis is the relation
of \(\mathscr{F}\) to the operator \(M\) of multiplication by the
independent variable:
\begin{equation*}%
Mx(t)=tx(t)\,
\end{equation*}
as well as to the operator \(D\) of differentiation with respect
to the independent variable
\begin{equation*}%
Dx(t)=\frac{dx(t)}{dt}\,.
\end{equation*}
Let \(x(t)\) be a smooth function vanishing fast enough on the
infinity. Then, taking the derivative under the integral, we
obtain.
\begin{equation*}%
\frac{d\,\,}{dt}\int\limits_{-\infty}^{\infty}x(\xi)e^{it\xi}\,d\xi=
i\int\limits_{-\infty}^{\infty}\xi{}x(\xi)e^{it\xi}\,d\xi\,.
\end{equation*}
Integrating by parts, we obtain
\begin{equation*}%
\int\limits_{-\infty}^{\infty}\frac{dx(\xi)}{d\xi}\,e^{it\xi}\,d\xi=
-it\int\limits_{-\infty}^{\infty}\,x(\xi)e^{it\xi}\,d\xi\,.
\end{equation*}
The last two equalities can be symbolically written as
\begin{subequations}
\label{CRL}
\begin{gather}
\label{CRL.a}
D\mathscr{F}=i\mathscr{F}M,\\
\label{CRL.b}%
 M\mathscr{F}=i\mathscr{F}D\,.
\end{gather}
\end{subequations}

The equalities \eqref{CRL} prompt us to use operators \(D\) and
\(M\) as building blocks for construction of differential
operators commuting with \(\mathscr{F}\).

 The attempt to find the operator \(L\) commuting with
\(\mathscr{F}\) in the form of a linear combination of \(D\) and
\(M\): \(L=\alpha{}D+\beta{}M\), where \(\alpha\) and \(\beta\)
are numbers, does not yield success. Multiplying \eqref{CRL.b}
with \(\alpha\), \eqref{CRL.a} with \(\beta\) and adding, we
obtain
\[\mathscr{F}(\alpha{}D+\beta{}M)=i(\beta{}D+\alpha{}M)\mathscr{F}\,.\]
If we assume that \(\mathscr{F}L=L\mathscr{F}\) for \(L\) of the
form \(L=\alpha{}D+\beta{}M\), that is
\[\mathscr{F}(\alpha{}D+\beta{}M)=(\alpha{}D+\beta{}M)\mathscr{F},\]
then
\[(\alpha{}D+\beta{}M)\mathscr{F}=i(\beta{}D+\alpha{}M)\mathscr{F}\]
with these \(\alpha\) and \(\beta\). Since the operator
\(\mathscr{F}\) is invertible, we may cancel on \(\mathscr{F}\):
\[(\alpha{}D+\beta{}M)=i(\beta{}D+\alpha{}M)\,.\]
From the last equality it follows that
\[\alpha=i\beta, \quad \beta=i\alpha\,.\]
This homogeneous linear system has only trivial solution:
\mbox{\(\alpha=0,\,\beta=0\).}

However, the commutation equation od the form
\begin{equation}%
\label{EVP}%
 \mathscr{F}L=\kappa{}L\mathscr{F},\ \ \ L=\alpha{}D+\beta{}M,\ \
\textup{where}\ \ \kappa,\,\alpha,\,\beta\ \ \textup{are numbers},
\end{equation}%
is solvable. (This equation can be considered as an eigenvalue
problem. The number \(\kappa\) plays the role of an eigenvalue,
the pair \((\alpha,\,\beta)\) forms an eigenvector.) Substituting
\eqref{CRL} into \eqref{EVP}, we come to the equalities
\begin{equation}
\label{CEP}%
-i\beta=\kappa\alpha, \quad -i\alpha=\kappa\beta\,.
\end{equation}
The system \eqref{CEP} is has two eigenvalues: \(\kappa=i\) and
\(\kappa=-i\). The eigenvector corresponding to \(\kappa=i\) is:
\(\alpha=-1,\,\beta=1\), the eigenvector corresponding to
\(\kappa=i\) is: \(\alpha=1,\,\beta=1\). Let us introduce the
linear combinations \(\alpha{}D+\beta{}M\) corresponding to these
eigenvectors \(\alpha,\,\beta\):
\begin{subequations}
\label{CAO}%
\begin{align}
\label{CAO1}%
\mathfrak{a}^{\dag}&=-\frac{d\,\,}{dt}+t\,,\\
\label{CAO2}%
\mathfrak{a}^{\phantom{\dag}}&=\phantom{+}\frac{d\,\,}{dt}+t\,.
\end{align} The last two formulas may be rewritten  in the
following way
\begin{equation}
\label{CAO3}%
(\mathfrak{a}^{\dag}x)(t)=
-e^{\frac{t^2}{2}}\frac{d\,}{dt}\Big(e^{-\frac{t^2}{2}}x(t)\Big),\quad
(\mathfrak{a}x)(t)=e^{-\frac{t^2}{2}}\frac{d\,}{dt}%
\Big(e^{\frac{t^2}{2}}x(t)\Big)\,.
\end{equation}
\end{subequations}
The differential operator \(\mathfrak{a}^{\dag}\) is said to be
the \emph{creation operator};\\ the  differential operator
\(\mathfrak{a}\) is said to be the \emph{annihilation operator}.\\
(The terminology comes from  the quantum physics.)

Here we consider the operators \(\mathfrak{a}^{\dag}\) and
\(\mathfrak{a}\) as formal differential expressions and do not
discuss their domains of definition. The algebraic relations in
which these operators are involved is all what is important for
us. In this stage of consideration we may assume that the domains
of definition \(\mathcal{D}_{\mathfrak{a}^{\dag}}\) and
\(\mathcal{D}_{\mathfrak{a}}\) of the operators
\(\mathfrak{a}^{\dag}\) and \(\mathfrak{a}\) coincides with the
Schwartz space \(\mathscrn{S}(\mathbb{R})\):
\(\mathcal{D}_{\mathfrak{a}^{\dag}}=\mathcal{D}_{\mathfrak{a}}=\mathscrn{S}(\mathbb{R})\,.\)
\begin{definition}%
\label{DefSchSp}
 The Schwartz space  \(\mathscrn{S}(\mathbb{R})\)
consists of all functions \(x(t)\) which are defined and
infinitely differentiable on the whole real axis and satisfy the
condition
\begin{equation}%
\label{SchSp}%
\sup\limits_{t\in\mathbb{R}}\left|t^p\frac{d^px(t)}{dt}\right|<\infty
\quad{}\text{for every integer\ \ }p\geq{}0,\,q\geq{}0\,.
\end{equation}%
\end{definition}

The operator \(\mathscr{F}\) maps the set
\(\mathscrn{S}(\mathbb{R})\) onto itself,
\(\mathscr{F}\mathscrn{S}(\mathbb{R})=\mathscrn{S}(\mathbb{R})\),
the operators  \(\mathfrak{a}^{\dag}\) and \(\mathfrak{a}\) map
\(\mathscrn{S}(\mathbb{R})\) into itself:
\[\mathscr{F}\mathscrn{S}(\mathbb{R})=\mathscrn{S}(\mathbb{R}),\ \
\mathfrak{a}^{\dag}\mathscrn{S}(\mathbb{R})\subseteq\mathscrn{S}(\mathbb{R}),\
\
\mathfrak{a}\mathscrn{S}(\mathbb{R})\subseteq\mathscrn{S}(\mathbb{R})\,.
\]

 The operators  \(\mathfrak{a}^{\dag}\) and \(\mathfrak{a}\)
satisfy the fundamental commutational relation
\begin{equation}
\label{FCR}
\mathfrak{a}\mathfrak{a}^{\dag}-\mathfrak{a}^{\dag}\mathfrak{a}=2I,
\end{equation}
where \(I\) is the identity operator in the space of functions.

The operators \(\mathfrak{a}\) and \(\mathfrak{a}^{+}\) are
formally adjoint each to other: for every functions
\(x,y\in\mathscrn{S}(\mathbb{R})\) , the equality holds
\begin{equation}%
\label{FMA}%
\langle\mathfrak{a}{}x,y\rangle=\langle{}x,\mathfrak{a}^+y\rangle\,.
\end{equation}

 The equality \eqref{EVP}, applied to those
\(\kappa,\,\alpha,\,\beta\) which was found by solving the
equation \eqref{CEP}, leads to the commutational relations:
\begin{subequations}
\label{CCR}%
\begin{align}
\label{CCR1}%
\mathscr{F}\mathfrak{a}^{\dag}=\phantom{-}i\mathfrak{a}^{\dag}\mathscr{F}\,,\\
\label{CCR2}%
\mathscr{F}\mathfrak{a}^{\phantom{\dag}}=-i\mathfrak{a}^{\phantom{\dag}}\mathscr{F}\,.
\end{align}
\end{subequations}
From \eqref{CCR} it becomes clear how to construct the
differential expression which commutes (formally) with
\(\mathscr{F}\). Let
\begin{equation}%
\label{Wor}
W=(\mathfrak{a}^{\dag})^{k_1}\mathfrak{a}^{l_1}(\mathfrak{a}^{\dag})^{k_2}\mathfrak{a}^{l_2}
\cdot\,\cdots\,\cdot(\mathfrak{a}^{\dag})^{k_p}\mathfrak{a}^{l_p}
\end{equation}
be a "word" composed from the "letters"  \(\mathfrak{a}^{\dag}\)
and \(\mathfrak{a}\). From \eqref{CCR} it follows that
\begin{equation}%
\label{CrW}%
\mathscr{F}W=(-1)^{\varepsilon}W\mathscr{F},
 \end{equation}
 where
 \begin{math}\varepsilon=k_1+\,\cdots\,+k_p-l_1-\,\cdots\,-l_p\,.
 \end{math}
 Thus under the condition
 \[k_1+\,\cdots\,+k_p=l_1+\,\cdots\,+l_p\,\]
 the differential expression \(W\) defined by \eqref{Wor} commutes
 with the Fourier operator \(\mathscr{F}\). The order \(\textup{ord}\,W\)
 of the (formal) differential operator \(W\) is:
 \[\textup{ord}\,W=k_1+\,\cdots\,+k_p+l_1+\,\cdots\,+l_p\,.\]
 The least "non-trivial" value for \(\textup{ord}\,W\) is two. So,
 there exist two words \(W\) which present operators commuting
 with \(\mathscr{F}\): \(W=\mathfrak{a}\mathfrak{a}^{\dag}\)
 and \(W=\mathfrak{a}^{\dag}\mathfrak{a}\):
 \begin{subequations}
 \label{RNO}
 \begin{align}
 \label{RNO1}
\mathfrak{a}\mathfrak{a}^{\dag}&=-\frac{d^2\,\,}{dt^2}+t^2+I\,; \\
 \label{RNO12}
 \mathfrak{a}^{\dag}\mathfrak{a}&=-\frac{d^2\,\,}{dt^2}+t^2-I\,.
 \end{align}
 \end{subequations}
Therefore the differential expression
\begin{equation}%
\label{CDO}%
 L=-\frac{d^2\,\,}{dt^2}+t^2\,,
\end{equation}%
 commutes with
\(\mathscr{F}\):
\begin{equation}%
\label{OLCF}%
 \mathscr{F}L=L\mathscr{F}\,.
\end{equation}%
(We already know this even in more precise, non-formal
formulation: see Theorem 2.2). From the equalities
\begin{subequations}
 \label{RHO}
 \begin{align}
 \label{RHO1}
L=\mathfrak{a}\mathfrak{a}^{\dag}&-I\,\\
\label{RHO2} L =\mathfrak{a}^{\dag}\mathfrak{a}&+I
 \end{align}
 \end{subequations}
and from the fundamental commutational relation \eqref{FCR} it
follows that
\begin{subequations}
\label{CFCR}
 \begin{align}
\label{CFCR1}
L\mathfrak{a}^{\dag}-\mathfrak{a}^{\dag}L&=\phantom{-}2\mathfrak{a}^{\dag}\,,\\
\label{CFCR2}%
L\mathfrak{a}^{\phantom{\dag}}-\mathfrak{a}^{\phantom{\dag}}L&=-2\mathfrak{a}^{\phantom{\dag}}\,.
\end{align}
\end{subequations}
From \eqref{CFCR1} it follows that the operator
\(\mathfrak{a}^{\dag}\) is \emph{raising operator} with respect to
the operator \(L\), and from \eqref{CFCR2} it follows that the
operator \(\mathfrak{a}\) is \emph{lowering operator} with respect
to the operator \(L\). This means that if \(h\) is an eigenvector
of \(L\) corresponding to an eigenvalue \(\lambda\):%
 \begin{equation}%
 \label{EV}
 Lh=\lambda{}h,\ \ h\not=0,
 \end{equation}
then
\begin{subequations}
\label{LEV}%
\begin{equation}%
\label{LEV1}%
L(\mathfrak{a}^{\dag}h)=(\lambda+2)(\mathfrak{a}^{\dag}h)\,,
\end{equation}
and
 \begin{equation}%
\label{LEV2}%
 L(\mathfrak{a}h)=(\lambda-2)(\mathfrak{a}h)\,.
 \end{equation}
\end{subequations}
Thus,  then the vector \(\mathfrak{a}^{\dag}h\) is an eigenvector
of \(L\) corresponding the eigenvalue \(\lambda+2\) if
\(\mathfrak{a}^{\dag}h\not=0\). Analogously,
 the vector \(\mathfrak{a}h\) is an eigenvector of \(L\)
corresponding to the eigenvalue \(\lambda-2\) if
\(\mathfrak{a}h\not=0\). So, the operator \(\mathfrak{a}^{\dag}\)
increases, and the operator \(\mathfrak{a}\) decreases the
eigenvalue of the operator \(L\). Rising and lowering operators
are collectively known as \emph{ladder operators}.

From \eqref{CFCR} and \eqref{FCR} one deduces by induction that
\begin{subequations}
\label{CFCRn}
 \begin{alignat}{3}
\label{CFCRn1}
L(\mathfrak{a}^{\dag})^n&-&(\mathfrak{a}^{\dag})^nL&=
\phantom{-}2n(\mathfrak{a}^{\dag})^n\,&,\quad &n=0,\,1,\,2,\,\ldots\,,\\
\label{CFCRn2}%
L\mathfrak{a}^n\ \ &-&\mathfrak{a}^nL\ \ &=-2n\mathfrak{a}^n\,,&\
\quad &n=0,\,1,\,2,\,\ldots\,.
\end{alignat}
\end{subequations}
If \(h\) is an eigenvector of \(L\) corresponding the eigenvalue
of \(\lambda\), then
\begin{subequations}
\label{LEVn}%
\begin{equation}%
\label{LEVn1}%
L(\mathfrak{a}^{\dag})^nh=(\lambda+2n)(\mathfrak{a}^{\dag})^nh\,,
\ \quad n=0,\,1,\,2,\,\ldots\,.
\end{equation}
and
 \begin{equation}%
\label{LEVn2}%
 L\mathfrak{a}^nh=(\lambda-2n)\mathfrak{a}^nh\,,\quad
 n=0,\,1,\,2,\,\ldots\,.
 \end{equation}
\end{subequations}
\begin{lemma}
\label{GsLe}{\ } \\ %
\hspace*{2.0ex}\textup{1.}The equation \(\mathfrak{a}^{\dag}h=0\)
has no non-zero solutions
 from \(L^2(\mathbb{R})\).\\
\hspace*{2.0ex}\textup{2.}Every solution of the equation
\(\mathfrak{a}h=0\) is of the form \(h(t)=ch_0(t)\), where \(c\)
is a constant, and
\begin{equation}%
\label{GrS}%
h_0(t)=e^{-\frac{t^2}{2}}\,.
\end{equation}%
\end{lemma}
\begin{proof}
 The equation \(\mathfrak{a}h=0\) is
the differential equation
\[\frac{dh}{dt}+th(t)=0\,.\]
Every solution \(h(t)\) of this equation is of the form
\(h(t)=ch_0(t)\). The equation \(\mathfrak{a}^{\dag}h=0\) is the
differential equation
\[-\frac{dh}{dt}+th(t)=0\,.\]
Every solution \(h(t)\) of this equation is of the form
\(h(t)=ce^{\frac{t^2}{2}}\), and does not belong to
\(L^2(\mathbb{R})\) if  is not identically zero.
\end{proof}
Now we are in position to do  the spectral analysis of the
operator \(L\). Let us denote
\begin{equation}%
\label{LHY}%
 h_n(t)=(\mathfrak{a}^{\dag})^nh_0(t),\quad
n=0,\,1,\,2,\,\ldots\,.
\end{equation}%
By Lemma \ref{GsLe}, \(h_n\not\equiv{}0,\ \
n=0,\,1,\,2,\,\ldots\). Using the explicit expression \eqref{CAO1}
for the operator \(\mathfrak{a}^{\dag}\), we obtain by induction
that%
\begin{subequations}
\label{HEV}
\begin{equation}%
\label{HEV1}
 h_n(t)=H_n(t)e^{-\frac{t^2}{2}},\,
\end{equation}%
 where
\begin{equation}%
\label{HEV2}
 H_n(t)=t^n+   \,\cdots
\end{equation}%
\end{subequations}
is an unitary polynomial of degree \(n\). Thus,
\[h_n(t)\in{}L^2(\mathbb{R}).\]
A little bit different representation of \(h_n\) can be obtained
from \eqref{LHY} and \eqref{CAO3}:
\begin{equation}
\label{AnRe}
h_n(t)=e^{\frac{t^2}{2}}\bigg(\frac{d\,\,}{dt}\bigg)^ke^{-t^2}\,.
\end{equation}
 From \eqref{RHO2} and
\(\mathfrak{a}h_0=0\) it follows that
\begin{equation}%
\label{SGSE}%
 Lh_0=h_0\,.
\end{equation}%
From \eqref{SGSE} and \eqref{LEVn1} it follows that for
\(n=0,\,1,\,2,\,\ldots\,,\)
\[Lh_n=\lambda_n{}h_n,\]
where
\begin{equation}%
\label{ECEV}%
 \lambda_n=2n+1,\quad n=0,\,1,\,2,\,\ldots\,.
 \end{equation}%

Thus we have constructed the sequence of eigenfunctions
\eqref{LHY}-\eqref{HEV} of the differential operator \(L\). The
operator \(L\) is defined by the differential expression
\eqref{CDO} on the set \(\mathscrn{S}(\mathbb{R})\) defined above
in Definition \ref{DefSchSp}.

The operator \(L\) defined on the set \(\mathscrn{S}(\mathbb{R})\)
is symmetric with respect to the standard scalar product
\(\langle\,.\,,\,.\,\rangle\) in \(L^2(\mathbb{R})\):
\begin{equation}%
\label{symm}%
\langle{}Lx,y\rangle=\langle{}x,L{}y\rangle\quad\forall
x,y\in\mathscrn{S}(\mathbb{R})\,.
\end{equation}%
 Every eigenfunction
\(h_n(t)\), \eqref{HEV}, belongs to the space
\(\mathscrn{S}(\mathbb{R})\), and all eigenvalues \(\lambda_n\),
\eqref{ECEV}, are pairwise different. These properties ensure that
the eigenfunctions \(h_n\) are pairwise orthogonal:
\begin{equation}%
\langle{}h_p,h_q\rangle=0,\quad\forall \ \
0\leq{}p,q<\infty,\,p\not=q\,.
\end{equation}%

\begin{lemma}%
\label{Compl}%
The system \(\lbrace{}h_n\rbrace_{0\leq{}n<\infty}\) is complete
in the space \(L^2(\mathbb{R})\),  i.e. the linear span of this
system is a set which is dense in \(L^2(\mathbb{R})\).
\end{lemma}
\begin{proof}
From \eqref{HEV} it follows that
\[\textup{span}\,\lbrace{}h_n(t)\rbrace_{0\leq{}n<\infty}=
\textup{span}\,\lbrace{}t^ne^{\frac{t^2}{2}}\rbrace_{0\leq{}n<\infty}\,.\]
If the system \(\lbrace{}h_n\rbrace_{0\leq{}n<\infty}\) is not
complete in \(L^2(\mathbb{R})\), then the system
\(\lbrace{}t^ne^{\frac{t^2}{2}}\rbrace_{0\leq{}n<\infty}\) also is
not complete in the space \(L^2(\mathbb{R})\). Then there exists a
function \(\varphi\in{}L^2(\mathbb{R})\) which is orthogonal to
every \(t^ne^{-\frac{t^2}{2}}\),  that is
\begin{equation}%
\label{OrtCo}%
\int\limits_{-\infty}^\infty\varphi(t)\,t^ne^{-\frac{t^2}{2}}\,dt=0\
\ \forall n=0,\,1,\,2,\,\ldots\,.
\end{equation}%
 Let us consider the function
\begin{equation*}%
f(z)=\int\limits_{-\infty}^\infty\varphi(t)e^{-\frac{t^2}{2}}e^{izt}\,dt\,.
\end{equation*}%
Since  \(\varphi\in{}L^2(\mathbb{R})\), the above integral exists
for every \(z\in\mathbb{C}\) and represents a function holomorphic
in the whole complex plane \(\mathbb{C}\). The orthogonality
conditions \eqref{OrtCo} that all derivatives \(f^{(n)}\) of the
function \(f\) vanish at the origin: \(f^{(n)}(0)=0,\
n=0,\,1,\,2,\,\ldots\,\). By the uniqueness theorem for analytic
functions, \(f(z)\equiv{}0\). In particular,
 \(\int\limits_{-\infty}^\infty\varphi(t)e^{-\frac{t^2}{2}}e^{i\xi{}t}\,dt=0\)
for every \(\xi\in\mathbb{R}\). By the uniqueness theorem for the
Fourier integral, \(\varphi(t)e^{-\frac{t^2}{2}}\equiv0\), hence
\(\varphi(t)\equiv0\).
\end{proof}
A complete orthogonal system in a Hilbert space is an orthogonal
basis of this space. Thus the following result was established:
\begin{lemma}
\label{ortbas}%
The system \(\lbrace{}h_n(t)\rbrace_{\,0\leq{}n<\infty}\),
\eqref{LHY}-\eqref{HEV}, is an orthogonal basis of the space
\(L^2(\mathbb{R}).\)
\end{lemma}

It turns out that we have find \emph{all} eigenfunctions of \(L\).
There is no eigenfunctions essentially other then \(h_n\). To make
the formulation precise we have to define accurately the
appropriate operator focusing our attention on its domain of
definition.\\[0.0ex]
\begin{definition}
\label{DDL}%
 Domain of definition \(\mathcal{D}_{\mathscr{L}}\) of
the operator \(\mathscr{L}\) consists of functions \(x(t)\) which
are defined at every point of the real axis \(\mathbb{R}\) and
satisfy the
conditions:\\[1.5ex]
\hspace*{2.0ex}\textup{1.}\ \ %
\begin{minipage}[t]{0.94\linewidth}
 \(x\) is differentiable at every point \(t\in\mathbb{R}\), its
derivative \(\dfrac{dx(t)}{dt}\) is absolutely continuous on every
finite subinterval of \(\mathbb{R}\);
\end{minipage}\\[1.0ex]
\hspace*{2.0ex}\textup{2.}\ \ %
\begin{minipage}[t]{0.94\linewidth}
\(x(t)\in{}L^2(\mathbb{R});\)
\end{minipage}\\[1.0ex]
\begin{minipage}[t]{0.94\linewidth}
\hspace*{2.0ex}\textup{3.}\ \ %
\(-\dfrac{d^2x(t)}{dt^2}+t^2x(t)\in{}L^2(\mathbb{R});\)
\end{minipage}\\[1.5ex]
 The action of the operator
\(\mathscr{L}\) on \(x\) is defined as:
\begin{equation}
\label{AcHO}%
 (\mathscr{L}x)(t)=-\dfrac{d^2x(t)}{dt^2}+t^2x(t)\quad
 \textup{for}\ \ x\in\mathcal{D}_{\mathscr{L}}\,.
\end{equation}%
\end{definition}
\begin{lemma}%
\label{VBC}%
Let a function \(x(t)\) belongs to the domain of definition \(\)
of the operator \(\).  \textup{(}See \textup{Definition
\ref{DDL})}.

 Then
 \begin{equation}
 \label{VaBoCo}
 \lim_{t\to\pm\infty}x(t)=0,\quad
 \lim_{t\to\pm\infty}\frac{dx(t)}{dt}=0\,.
 \end{equation}
\end{lemma}%
\begin{proof} The function %
\(\bigg(\!\!-\dfrac{d^2x(t)}{dt^2}+t^2x(t)\bigg)\cdot\overline{x(t)}\)
belongs to \(L^1(\mathbb{R})\) since each of the factors of this
product belongs to \(L^2(\mathbb{R})\). Therefore, there exists
the finite limit
\begin{equation*}%
\lim_{\substack{a\to-\infty\\b\to+\infty}}\int\limits_{a}^{b}
\bigg(\!\!-\dfrac{d^2x(t)}{dt^2}+t^2x(t)\bigg)\,\overline{x(t)}\,dt\,.
\end{equation*}
Integrating by parts, we obtain
\begin{multline*}%
\int\limits_{a}^{b}
\bigg(\!\!-\dfrac{d^2x(t)}{dt^2}+t^2x(t)\bigg)\,\overline{x(t)}\,dt=\\
-\frac{dx(t)}{dt}\,\overline{x(t)}\,\bigg|_{t=a}^{t=b}+
\int\limits_{a}^{b}\bigg(\bigg|\frac{dx(t)}{dt}\bigg|^2+t^2\big|x(t)\big|^2\bigg)\,dt\,.
\end{multline*}
For fixed \(a\), there exists the limit %
\(\lim_{b\to\infty}\int\limits_{a}^{b}%
\bigg(\bigg|\dfrac{dx(t)}{dt}\bigg|^2+t^2\big|x(t)\big|^2\bigg)\,dt\),
finite or infinite. Hence, there exists the limit
\[\lim_{b\to+\infty}\frac{d|x(t)|^2}{dt}\,_{\big|_{t=b}}=
\lim_{b\to+\infty}\bigg(\frac{dx(t)}{dt}\,\overline{x(t)}+
\overline{\frac{dx(t)}{dt}}\,x(t)\bigg)_{\big|_{t=b}}\,,\] finite
or infinite. The last limit must be equal to zero, otherwise
\(\lim_{t\to\infty}|x(t)|=+\infty\), and the condition
\(x\in{}L^2(\mathbb{R})\) shall be violated. Since %
\(\lim_{t\to\infty}\frac{dx(t)}{dt}=0\) and
\(x\in{}L^2(\mathbb{R})\), the condition
\(\lim_{t\to\infty}x(t)=0\) will be satisfied.
\end{proof}
\begin{lemma}
\label{SymmOpe}%
 The operator \(\mathscr{L}\) is a symmetric
operator in \(L^2(\mathbb{R})\), that is
\begin{equation}%
\label{SymmOp}%
\langle\mathscr{L}x,y\rangle=\langle{}x,\mathscr{L}y\rangle \quad
\ \forall\,x,\,y\in\mathcal{D}_{\mathscr{L}}\,.
\end{equation}%
\end{lemma}
\begin{proof} The equality \eqref{SymmOp} is obtained
integrating by part twice on fixed finite interval and then
passing to the limit as the endpoints of this interval tend to
\(\mp\infty\). The term out the integral vanish according to Lemma
\ref{VBC}.
\end{proof}
\begin{lemma}%
\label{MEV}%
If a number \(\lambda\) is an eigenvalue of the operator
\(\mathscr{L}\), then the multiplicity of \(\lambda\) is equal to
one.
\end{lemma}%
\begin{proof}
Let \(u(t)\) and \(v(t)\) are eigenfunction of \(\mathscr{L}\)
corresponding to the same eigenvalue \(\lambda\):
\begin{equation*}%
u\in\mathcal{D}_{\mathscr{L}},\ \ v\in\mathcal{D}_{\mathscr{L}},
\mathscr{L}u=\lambda{}u,\ \mathscr{L}v=\lambda{}v\,.
\end{equation*}%
 Since both functions, \(u(t)\) and \(v(t)\) are solutions of the
 same differential equation of the form
 \[-\frac{d^2x(t)}{dt^2}+t^2x(t)-\lambda{}x(t)=0\,,\]
the Wronskian \(W(t)=\dfrac{du(t)}{dt}v(t)-u(t)\dfrac{dv(t)}{dt}\)
does not depend on \(t\). According to Lemma \ref{VBC}, \(
\lim_{t\to\pm\infty}W(t)=0\). Hence \(W(t)\equiv0\), and the
functions \(u(t)\) and \(v(t)\) are proportional.
\end{proof}
\begin{lemma}
\label{AllEF}%
Let \(h(t)\not\equiv0\) be an eigenfunction of differential operators %
\(\mathscr{L}\),\textup{(}See \textup{Definition \ref{DDL})},
corresponding an eigenvalue \(\lambda\):
\[h\in\mathcal{D}_{\mathscr{L}},\ %
\mathscr{L}=\lambda{}h.\]

Then \(\lambda\) coincides with one of the eigenvalues
 \eqref{ECEV}: \(\lambda=\lambda_{n_0}\) for some \(n_0\),
 and the function \(h(t)\) is
proportional to the appropriate eigenfunction:
\(h(t)=ch_{n_0}(t)\), where \(c\) is a constant.
\end{lemma}
\begin{proof}
Combining the equalities \[Lh_n=\lambda_nh_n, \ \ Lh=\lambda{}h,\
\ \langle{}Lh_n,h\rangle=\langle{}h_n,Lh\rangle ,\] we obtain that
\[(\lambda_n-\lambda)\langle{}h_n,h\rangle=0\,.\]
If \(\lambda\) differs from any of eigenvalues \(\lambda_n\), then
\(\langle{}h_n,h\rangle=0\ \forall\,n\). Since the system
\(\lbrace{}h_n\rbrace_{0\leq{}n<\infty}\) is complete in
\(L^2(\mathbb{R})\), and \(h\not=0\), this is impossible. Thus,
\(\lambda=\lambda_{n_0}\) for some \(n_0\). By Lemma \ref{VBC},
the eigenvalue \(\lambda_{n_0}\) is of multiplicity one. Thus, the
eigenfunction \(h\) is proportional to the eigenfunction
\(h_{n_0}\).
\end{proof}
 We already know, \eqref{OLCF}, see also \textbf{Theorem 2.2}, that the differential operator
 \(\mathscr{L}\) commutes with the operator \(\mathscr{F}\):
 \begin{equation}
 \label{RCR}
\mathscr{L}\mathscr{F}x= \mathscr{F}\mathscr{L}x\quad
 \forall\,\,x\in\mathscrn{S}(\mathbb{R})\,.
 \end{equation}
 The eigenfunctions \(h_n\), \eqref{LHY}-\eqref{HEV}, belongs the
 Schwartz space \(\mathscrn{S}(\mathbb{R})\). Thus, the equality
 \eqref{RCR} is applicable to \(x=h_n\):
 \begin{equation}%
 \label{CRFE}%
 \mathscr{L}\mathscr{F}h_n=\mathscr{F}\mathscr{L}h_n\,.
 \end{equation}%
Since \(h_n\) is an eigenvector of \(\mathscr{L}\) corresponding
to the eigenvalue \(\lambda_n\):
\[\mathscr{L}h_n=\lambda_nh_n\,,\]
we see from \eqref{CRFE}, that
\begin{equation}%
 \label{CRFEF}%
 \mathscr{L}(\mathscr{F}h_n)=\lambda_n(\mathscr{F}h_n)\,.
 \end{equation}%
 The last equality means that the vector \(\mathscr{F}h_n\) also
 is an eigenvector of \(\mathscr{L}\) corresponding
to the eigenvalue \(\lambda_n\). Since eigenvalues of the operator
\(\mathscr{L}\) are of multiplicity one, (Lemma \ref{MEV}), The
vector \(\mathscr{F}h_n\) is proportional  to the vector \(h_n\):
\begin{equation}
\label{EVFP}%
\mathscr{F}h_n=\mu_nh_n\,.
\end{equation}
However the operator \(\mathscr{F}\) is unitary,
\begin{equation}
\label{UnF}%
 \mathscr{F}^{\ast}\mathscr{F}=I\,,
\end{equation}%
and
\begin{equation}
\label{SqF}%
 \mathscr{F}^2=-I.
\end{equation}
The last equality is a consequence of the unitarity of
\(\mathscr{F}\), \eqref{UnF}, and the relation
\begin{equation*}%
 (\mathscr{F}^{\ast}x)(t)=(\mathscr{F}x)(-t)\quad \forall\,x.
\end{equation*}%
From \eqref{SqF}, the spectral mapping theorem and unitarity of
\(\mathscr{F}\) it follows that the eigenvalues of \(\mathscr{F}\)
are among the numbers
\(\lambda=1,\,\lambda=-1,\,\lambda=i,\,\lambda=-1\), and that the
space \(L^2(\mathbb{R})\) is the orthogonal sum of the appropriate
eigenspaces:
\begin{subequations}
\label{Osum}
\begin{gather}
\label{Osum1}
L^2(\mathbb{R})=\mathcal{X}_{\,1}\oplus\mathcal{X}_{\,-1}\oplus%
\mathcal{X}_{\,i}\oplus\mathcal{X}_{\,-i}\,,\\
\label{Osum2}%
 \mathscr{F}x=\lambda{}x\ \ \textup{for} \ \
x\in\mathcal{X}_{\,\lambda},\ \ \lambda=1,\,-1,\,i,\,-i\,.
\end{gather}
Thus, in \eqref{EVFP} \(\mu_n\) may take only one of four possible
values: \(1\), \(-1\), \(i\), \(-i\).
\end{subequations}
For \(n=0\), the equality means
\[\frac{1}{\sqrt{2\pi}}\int\limits_{-\infty}^{\infty}%
e^{-\frac{\xi^2}{2}}e^{i\xi{}t}\,d\xi=\mu_0\,e^{-\frac{t^2}{2}}\,.\]
Setting \(t=0\) in the last equality, we obtain
\[\mu_0=\frac{1}{\sqrt{2\pi}}\int\limits_{-\infty}^{\infty}\,%
e^{-\frac{\xi^2}{2}}\,d\xi\] In particular, \(\mu_0>0\). Hence,
\begin{equation}
\mu_0=1\,.
\end{equation}
From \eqref{CCR1} we obtain that
\begin{equation}%
\label{nEV}%
\mathscr{F}(\mathfrak{a}^{\dag})^nx=i^n\,%
\mathscr{F}(\mathfrak{a}^{\dag})^nx\quad\forall
x\in\mathscrn{S}(\mathbb{R})\,.
\end{equation}
 (The last equality is a particular case of the equality \eqref{Wor}-%
 \eqref{CrW}.) Substituting \(x=h_0\) into \eqref{nEV} and taking
 into account \eqref{LHY}, we see that
\begin{subequations}
\label{nEV0}
 \begin{equation}%
 \label{nEV1}
 \mathscr{F}h_n=i^nh_n\,,\quad n=0,\,1,\,2,\,\ldots\,\,.
 \end{equation}%
 So,
 \begin{equation}%
\label{nEV2}%
 \mu_n=i^n,\quad n=0,\,1,\,2,\,\ldots\,\,.
 \end{equation}%
\end{subequations}
Thus, we obtain the following result:
\begin{theorem}\label{BCES}{\ }\\%
\hspace*{2.0ex}\textup{1.}
\begin{minipage}[t]{0.94\linewidth}
The spectrum ow the Fourier operator \( \mathscr{F}\), considered
in the Hilbert space \(L^2(\mathbb{R})\), consists of four points:
\[\mu=1,\ \mu=-1,\ \mu=i,\ \mu=-i\,.\]
Each of this points is an eigenvalue of  infinite multiplicity.
\end{minipage}\\[1.5ex]%
\hspace*{2.0ex}\textup{2.}
\begin{minipage}[t]{0.94\linewidth}
The space \(L^2(\mathbb{R})\) is the orthogonal sum \eqref{Osum}
of the appropriate eigenspaces \(\mathcal{X}_\lambda\),
\(\lambda=1,\, i,\,-1,\,-i\):
\begin{gather*}
L^2(\mathbb{R})=\mathcal{X}_{\,1}\oplus\mathcal{X}_{\,-1}\oplus%
\mathcal{X}_{\,i}\oplus\mathcal{X}_{\,-i}\,,\\
\mathscr{F}x=\lambda{}x\ \ \textup{for} \ \
x\in\mathcal{X}_{\,\lambda},\ \ \lambda=1,\,-1,\,i,\,-i\,.
\end{gather*}
Each of these eigenspaces \(\mathcal{X}_{\,\lambda}\) is
infinite-dimensional.
\end{minipage}\\[1.5ex]%
\hspace*{2.0ex}\textup{3.}
\begin{minipage}[t]{0.94\linewidth}
The systems
 \begin{align}
 \lbrace{}h_{4k\phantom{+2}}\rbrace_{k=0,\,1,\,2,\,\ldots\,},\  &
\lbrace{}h_{4k+1}\rbrace_{k=0,\,1,\,2,\,\ldots\,}, \notag \\ %
\lbrace{}h_{4k+2}\rbrace_{k=0,\,1,\,2,\,\ldots\,},\ & %
\lbrace{}h_{4k+3}\rbrace_{k=0,\,1,\,2,\,\ldots\,},\
 \end{align}
form orthogonal bases in the eigenspace \(\mathcal{X}_{1}\),
\(\mathcal{X}_{i}\), \(\mathcal{X}_{-1}\), \(\mathcal{X}_{-i}\)
respectively.
\end{minipage}\\[2.0ex]
\end{theorem}%

The system \(\lbrace{}h_{n}\rbrace_{n=0,\,1,\,2,\,\ldots\,}\) is
not normalized. Let us normalize this system. By induction with
respect to \(n\) we derive from \eqref{FCR} and \eqref{LHY} that
\begin{equation}
\label{Bac}%
 \mathfrak{a}h_n=2nh_{n-1},\  n=1,\,2,\,\ldots\,;\  \ %
\mathfrak{a}h_0=0\,.
\end{equation}
From \eqref{LHY}, \eqref{FMA} and \eqref{Bac},
\[\langle{}h_n,h_n\rangle=\langle\mathfrak{a}^{\dag}h_{n-1},h_n\rangle=
\langle{}h_{n-1},\mathfrak{a}h_n\rangle=2n
\langle{}h_{n-1},h_{n-1}\rangle\,.\] Iterating the equality
\(\langle{}h_n,h_n\rangle=2n\langle{}h_{n-1},h_{n-1}\rangle\),
 we obtain
\[\langle{}h_n,h_n\rangle=2^nn!\langle{}h_0,h_0\rangle\,.\]
By direct calculation, \(\langle{}h_n,h_n\rangle=\pi\). Thus,
\begin{equation}
\label{norm} %
\langle{}h_n,h_n\rangle=\pi2^nn!\,,\quad
n=0,\,1,\,2,\,\ldots\,.
\end{equation}
Let us introduce the normalized vectors
\begin{equation}
\label{normv} %
e_n=\pi^{-1/2}2^{-\frac{n}{2}}(n!)^{-1/2}h_n\,,\quad
n=0,\,1,\,2,\,\ldots\,.
\end{equation}
The system \( \lbrace{}e_{n}\rbrace_{n=0,\,1,\,2,\,\ldots\,}\)
forms an \emph{orthonormal} basis of the space
\(L^2(\mathbb{R})\), and
\begin{equation}%
\label{eigve}%
\mathscr{L}e_n=\lambda_ne_n, \ \ \ %
\mathscr{F}e_n=\mu_ne_n, \ \ \ n=0,\,1,\,2,\,\ldots\,,
\end{equation}%
where
\begin{equation}%
\label{eigva}%
 \lambda_n=2n+1,\quad \mu_n=i^{n}\,.
\end{equation}

The domain of definition \(\mathcal{D}_{\mathscr{L}}\) of the
operator \(\mathscr{L}\) can be described as follows:
\begin{lemma}
\label{DDO}%
 A vector \(x\) from \(L^2(\mathbb{R})\) belongs to the domain of
 definition \(\mathcal{D}_{\mathscr{L}}\) of the operator
 \(\mathscr{L}\) which was introduced in \textup{Definition \ref{DDL}} if
 and only if the coefficients \(c_n\) of the expansion of \(x\)
 in the Fourier series
 \begin{equation*}%
 x=\sum\limits_{0\leq{}n<\infty}c_ne_n
 \end{equation*}%
 with respect the orthonormal basis  \( \lbrace{}e_{n}\rbrace_{n=0,\,1,\,2,\,\ldots\,}\)
 satisfy the condition
 \[\sum\limits_{0\leq{}n<\infty}\lambda_n^2|c_n|^2<\infty\,.\]
 The vector \(\mathscr{L}x\) is represented by the ortogonal
 series
 \[\mathscr{L}x=\sum\limits_{0\leq{}n<\infty}\lambda_nc_ne_n\,.\]
\end{lemma}

If %
\begin{math}%
x=\sum\limits_{0\leq{}n<\infty}c_n(x)e_n
\quad\textup{and}\quad
\mathscr{F}x\sum\limits_{0\leq{}n<\infty}c_n(\mathscr{F}x)e_n
\end{math}%
are the Fourier expansions of the vectors \(x\) and
\(\mathscr{F}x\) with respect to the orthonormal basis coomposed
of the normalised eigenvectors of \(\mathscr{F}\), then
\(c_n(\mathscr{F}x)=i^nc_n(x) \). In particular,
\(|c_n(x)|=|c_n(\mathscr{F}x|)\).

From Lemma \eqref{DDO} wi derive the following theorem, which is
\(L^2\)-version of Theorem 2.2:
\begin{theorem}
\label{RigComRel}%
Let \(\mathscr{L}\) be the differential operator which was
introduced in \textup{Definition \ref{DDL}.}\\[1.5ex]%
\hspace*{2.0ex}\textup{1.}
\begin{minipage}[t]{0.94\linewidth}
A function \(x\) belongs the domain of definition
\(\mathcal{D}_{\mathscr{L}}\) of the operator \(\mathscr{L}\) if
and only if the function \(\mathscr{F}x\) belongs to
\(\mathcal{D}_{\mathscr{L}}\).
\end{minipage}\\[1.5ex]
\hspace*{2.0ex}\textup{2.}
\begin{minipage}[t]{0.94\linewidth}
If \(x\in\mathcal{D}_{\mathscr{L}}\), then
\[\mathscr{F}(\mathscr{L}x)=\mathscr{L}(\mathscr{F}x)\,.\]
\end{minipage}
\end{theorem}

Theorem \ref{BCES} gives a characterization of vectors from the
eigenspaces \(\mathcal{X}_\lambda\) of the Fourier operator
\(\mathscr{F}\) in terms of their Fourier coefficients  of these
vectors with respect to the orthogonal system
\(\lbrace{}h_{n}\rbrace_{n=0,\,1,\,2,\,\ldots\,}\). Recall that
this system  was introduced by
\eqref{LHY}-\eqref{GrS}-\eqref{CAO1}.

Below we present another characterization of the eigenspaces
\(\mathcal{X}_\lambda\) in the spirit of the work \cite{HaTi1} by
Hardy and Titchmarsh, where the description of the subspace
\(\mathcal{X}_1\) was done. (The results of  the works
\cite{HaTi1} and \cite{HaTi2} were reproduced in the book
\cite{Tit}.)

Functions \(x(t)\) which belongs to the eigenspace
\(\mathcal{X}_\lambda,\,\lambda=1,\,-1,\,i,\,-i\) of the Fourier
operator \(\mathscr{F}\) are described as the inverse Melline
transform
\begin{equation}
\label{IMT}%
 x(t)=\frac{1}{2\pi}\int\limits_{-\infty}^{\infty}%
 \varphi\bigg(\frac{1}{2}+i\eta\bigg)t^{-\big(\frac{1}{2}+i\eta\big)}\,d\eta\,,\quad
 0<t<\infty\,,
\end{equation}
where the form of the function \(\varphi(\zeta)\)
 depends on \(\lambda\). (Eigenfunctions \(x(t)\) are either even, if \(\lambda=\pm1\), or
 odd, if \(\lambda=\pm{}i\). Therefore it is enough to describe
 their restrictions on the positive half-axis only.)

 For \(\lambda=\pm1\), \(\varphi\) is of the form
 \begin{subequations}
  \label{DeES}
 \begin{equation}%
 \label{DeES1}
 \varphi\bigg(\frac{1}{2}+\zeta\bigg)
 =\psi(\zeta)\cdot2^{\frac{\zeta}{2}}\Gamma\bigg(\frac{1}{4}+\frac{\zeta}{2}\bigg)\,,
 \end{equation}%
 for \(\lambda=\pm{}i\), \(\varphi\) is of the form
 \begin{equation}%
 \label{DeES2}%
 \varphi\bigg(\frac{1}{2}+\zeta\bigg)
 =\psi(\zeta)\cdot2^{\frac{\zeta}{2}}\Gamma\bigg(\frac{3}{4}+\frac{\zeta}{2}\bigg)\,,
 \end{equation}%
 \end{subequations}
 where the function \(\psi\) is even if \(\lambda=1\) or
 \(\lambda=i\),  and the function \(\psi\) is odd if \(\lambda=-1\) or
 \(\lambda=-i\).

 Thus, for \(\lambda=\pm1\), the representation \eqref{IMT} takes
 the form
 \begin{subequations}
 \label{IMTCo}%
\begin{equation}
\label{IMTCo1}%
 x(t)=\frac{1}{2\pi}\int\limits_{-\infty}^{\infty}%
\psi(i\eta)\cdot2^{\frac{i\eta}{2}}\Gamma\bigg(\frac{1}{4}+\frac{i\eta}{2}\bigg)
 t^{-\big(\frac{1}{2}+i\eta\big)}\,d\eta\,,\quad
 0<t<\infty\,,
\end{equation}
and for \(\lambda=\pm{}i\), the representation \eqref{IMT} takes
 the form
\begin{equation}
\label{IMTCo2}%
 x(t)=\frac{1}{2\pi}\int\limits_{-\infty}^{\infty}%
\psi(i\eta)\cdot2^{\frac{i\eta}{2}}\Gamma\bigg(\frac{3}{4}+\frac{i\eta}{2}\bigg)
 t^{-\big(\frac{1}{2}+i\eta\big)}\,d\eta\,,\quad
 0<t<\infty\,,
\end{equation}
\end{subequations}

  The function \(\psi\), which is defined merely on the imaginary axis ,
  serves as a `free'
 parameter.  The only restrictions on \(\psi\) is its evenness (or oddness)
 and the convergence of the integral
\begin{equation}%
\label{CoCo}%
\int\limits_{\infty}^{\infty}\bigg|\varphi\bigg(\frac{1}{2}+i\eta\bigg)\bigg|^2%
\,d\eta<\infty\,.
\end{equation}
For \(\lambda=\pm1\), the condition \eqref{CoCo} takes the form
\begin{subequations}
\label{CoCoCo}%
\begin{equation}%
\label{CoCoCo1}%
\int\limits_{\infty}^{\infty}\big|\psi(i\eta)\big|^2
\bigg|\Gamma\bigg(\frac{1}{4}+i\frac{\eta}{2}\bigg)\bigg|^2%
\,d\eta<\infty\,,
\end{equation}
for \(\lambda=\pm{}i\), the condition \eqref{CoCo} takes the form
\begin{equation}%
\label{CoCoCo2}%
\int\limits_{\infty}^{\infty}\big|\psi(i\eta)\big|^2
\bigg|\Gamma\bigg(\frac{3}{4}+i\frac{\eta}{2}\bigg)\bigg|^2%
\,d\eta<\infty\,.
\end{equation}
\end{subequations}
\begin{theorem}{\ }\\[0.5ex]%
\label{AnDeES}%
\hspace*{2.0ex}\textup{1.}
\begin{minipage}[t]{0.94\linewidth}
Let the function \(x\) from \(L^2(\mathbb{R})\) belongs to the
eigenspace \(\mathcal{X}_1\) of the Fourier operator
\(\mathscr{F}\). Then the function \(x\) is representable in the
form \eqref{IMTCo1} on the positive half-axis \(\mathbb{R}_+\),
where \(\psi\) is an even function defined on the imaginary axis
 and satisfies the condition \eqref{CoCoCo1}.

\hspace{2.0ex}Conversely, if \(\psi\) is an even function defined
on the imaginary axis and satisfying the condition
\eqref{CoCoCo1}, and the function \(x(t)\) is defined by
\eqref{IMTCo1}  on the positive half-axis \(\mathbb{R}_+\) and is
extended from \(\mathbb{R}_+\) to \(\mathbb{R}\) as an even
function, then \(x\) belongs to the eigenspaces \(\mathcal{X}_1\)
of the operator \(\mathscr{F}\).
\end{minipage}\\[1.0ex]
\hspace*{2.0ex}\textup{2.}
\begin{minipage}[t]{0.94\linewidth}
Let the function \(x\) from \(L^2(\mathbb{R})\) belongs to the
eigenspace \(\mathcal{X}_{-1}\) of the Fourier operator
\(\mathscr{F}\). Then the function \(x\) is representable in the
form \eqref{IMTCo1} on the positive half-axis \(\mathbb{R}_+\),
where \(\psi\) is an odd function defined on the imaginary axis
 and satisfies the condition \eqref{CoCoCo1}.

\hspace{2.0ex}Conversely, if \(\psi\) is an odd function defined
on the imaginary axis and satisfying the condition
\eqref{CoCoCo1}, and the function \(x(t)\) is defined by
\eqref{IMTCo1}  on the positive half-axis \(\mathbb{R}_+\) and is
extended from \(\mathbb{R}_+\) to \(\mathbb{R}\) as an odd
function, then \(x\) belongs to the eigenspaces
\(\mathcal{X}_{-1}\) of the operator \(\mathscr{F}\).
\end{minipage}
\hspace*{2.0ex}\textup{3.}
\begin{minipage}[t]{0.94\linewidth}
Let the function \(x\) from \(L^2(\mathbb{R})\) belongs to the
eigenspace \(\mathcal{X}_{\,i}\) of the Fourier operator
\(\mathscr{F}\). Then the function \(x\) is representable in the
form \eqref{IMTCo2} on the positive half-axis \(\mathbb{R}_+\),
where \(\psi\) is an even function defined on the imaginary axis
 and satisfies the condition \eqref{CoCoCo2}.

\hspace{2.0ex}Conversely, if \(\psi\) is an even function defined
on the imaginary axis and satisfying the condition
\eqref{CoCoCo2}, and the function \(x(t)\) is defined by
\eqref{IMTCo2}  on the positive half-axis \(\mathbb{R}_+\) and is
extended from \(\mathbb{R}_+\) to \(\mathbb{R}\) as an even
function, then \(x\) belongs to the eigenspaces
\(\mathcal{X}_{\,i}\) of the operator \(\mathscr{F}\).
\end{minipage}\\[1.0ex]
\hspace*{2.0ex}\textup{4.}
\begin{minipage}[t]{0.94\linewidth}
Let the function \(x\) from \(L^2(\mathbb{R})\) belongs to the
eigenspace \(\mathcal{X}_{-i}\) of the Fourier operator
\(\mathscr{F}\). Then the function \(x\) is representable in the
form \eqref{IMTCo2} on the positive half-axis \(\mathbb{R}_+\),
where \(\psi\) is an odd function defined on the imaginary axis
 and satisfies the condition \eqref{CoCoCo2}.

\hspace{2.0ex}Conversely, if \(\psi\) is an odd function defined
on the imaginary axis and satisfying the condition
\eqref{CoCoCo2}, and the function \(x(t)\) is defined by
\eqref{IMTCo2}  on the positive half-axis \(\mathbb{R}_+\) and is
extended from \(\mathbb{R}_+\) to \(\mathbb{R}\) as an odd
function, then \(x\) belongs to the eigenspaces
\(\mathcal{X}_{-i}\) of the operator \(\mathscr{F}\).
\end{minipage}\\[1.5ex]
\hspace*{1.0ex}In each of the four cases, the correspondence
between \(x(t)\) and \(\psi(i\eta)\) is one-to-one. Moreover, the
Parseval identity holds, which is of the form
\begin{equation}
\label{RPI1}
\int\limits_{0}^{\infty}|x(t)|^2dt=\int\limits_{\infty}^{\infty}\big|\psi(i\eta)\big|^2
\bigg|\Gamma\bigg(\frac{1}{4}+i\frac{\eta}{2}\bigg)\bigg|^2%
\,d\eta
\end{equation}%
 in the cases \textup{1} or \textup{2}, and
of the form
\begin{equation*}
\int\limits_{0}^{\infty}|x(t)|^2dt=\int\limits_{\infty}^{\infty}\big|\psi(i\eta)\big|^2
\bigg|\Gamma\bigg(\frac{3}{4}+i\frac{\eta}{2}\bigg)\bigg|^2%
\,d\eta
\end{equation*}%
 in the cases \textup{3} or \textup{4}\,.
\end{theorem}%
Let us recall some well known facts:

 \textsf{The identities for the Euler Gamma-function
 \(\Gamma(\zeta)\)}:
\begin{alignat}{2}%
\label{Ga1}
\Gamma(\zeta+1)&=\zeta\Gamma(\zeta)\,,& &\quad\text{see\,\,\,%
\cite{WhWa}\,,\,\,\,\textbf{12.12},}\\[1.5ex]
\label{Ga2}
\Gamma(\zeta)\Gamma(1-\zeta)&=\dfrac{\pi}{\sin{}\pi\zeta}\,,\,\,\,&
&
\quad\text{see\,\,\,\cite{WhWa}\,,\,\,\,\textbf{12.14},}\\[1.5ex]
\label{Ga3}
\Gamma(\zeta)\Gamma\bigg(\zeta+\frac{1}{2}\bigg)&=2\sqrt{\pi}\,2^{-2\zeta}\Gamma(2\zeta),\,\,&
& \quad\text{see\,\,\,\cite{WhWa}\,,\,\,\,\textbf{12.15}.}
\end{alignat}

\textsf{The Melline transform.}\\ %
The Mellin transform is an integral transform that may be regarded
as the multiplicative version of the two-sided Laplace transform.
The Melline transform serves to link Dirichlet series with
automorphic functions. In particular, the inversion formula plays
a role in the proof of functional equation for Dirichlet series
similar to that for the Riemann zeta-function. See the article
"Melline Transform" on the page 192 of \cite[Volume 6]{ME} and
references there. See also the article "The Melline Trannsform" in
\textsf{Wikipedia}.

 Let \(f(t)\) be a function
defined for \(0\leq{}t<\infty\). \emph{The Melline transform }of
the function \(f(t)\) is said to be the function
\(\varphi(\zeta)\):
\begin{equation}
\label{MT}
\varphi(\zeta)=\int\limits_{0}^{\infty}t^{\,\zeta-1}f(t)\,dt\,.
\end{equation}
The Melline transform exists for those (complex) \(\zeta\) for
which the integral in \eqref{MT} exists. If the function \(f\) is
locally integrable on \((0,\infty)\), and for some real \(u,\,v:
u<v\), the estimates hold:
\begin{equation}%
\label{CFS}%
 f(t)=O(t^{-u})\ \ \textup{as} \ \ t\to{}+0,\quad
f(t)=O(t^{-v})\ \ \textup{as} \ \
t\to{}+\infty\,,%
\end{equation}%
 then the integral in \eqref{MT} exists for \(\zeta\) from the
 vertical strip\\
\(\big\lbrace\zeta:\,\,u<\textup{Re}\zeta<v\big\rbrace\), and the
inversion formula
\begin{equation}%
\label{IFMT}%
f(t)=\frac{1}{2\pi{}i}\int\limits_{c-i\infty}^{c+i\infty}t^{-\zeta}%
\,\varphi(\zeta)\,d\zeta\,,\quad 0<t<\infty\,,
\end{equation}%
 holds, where the
integration can be performed over any vertical straight line
\(\textup{Re}\,\zeta=c\) with \(u<c<v\). In particular, if
\(u<\frac{1}{2}<v\), then \(f\in{}L^2(0,\infty)\), the function
\(\varphi(\zeta)\) is defined on the vertical line
\(\textup{Re}\,\zeta=\frac{1}{2}\), and the Parseval equality
\begin{equation}%
\label{PIMT}%
 \int\limits_{0}^{\infty}|f(t)|^2\,dt=\frac{1}{2\pi}
\int\limits_{-\infty}^{\infty}\bigg|\varphi\bigg(\frac{1}{2}+i\eta\bigg)\bigg|^2d\eta
\end{equation}
 holds. The set of functions \(f(t)\) satisfying the
 conditions \eqref{CFS} for some \(u,v: u<\frac{1}{2}<v \),
  is dense in \(L^2(0,\infty)\). By a standard approximation reasoning,
which uses the Parseval identity \eqref{PIMT}, the Melline
transform \(\varphi(\zeta)\) can be defined for arbitrary function
\(f(t)\in{}L^2(0,\infty)\). This Melline transform
\(\varphi(\zeta)\) is defined as an \(L^2\)-function on the
vertical line
\(\big\lbrace\zeta:\,\textup{Re}\,\zeta=\frac{1}{2}\big\rbrace\).

Vice versa, let \(\varphi(\zeta)\) be an arbitrary function which
is holomorphic in the strip \(\big\lbrace\zeta:\,\,\alpha<
\textup{Re}\,\zeta<\beta\big\rbrace\), where
\(\alpha<\frac{1}{2}<\beta\), and satisfies the estimate
\(|\varphi(\zeta)|=O(|\zeta|^{-2})\) as \(|\zeta|\to\infty\) in
this strip. We  \emph{define} the function \(f(t)\) from
\(\varphi\) by \eqref{IFMT}, where the integral is taken over an
arbitrary line \(\lbrace\zeta:\,\,\textup{Re}\,\zeta=c\rbrace\)
with \(\alpha<c<\beta\). Then the function \(f\) satisfy the
estimates \eqref{CFS} with any arbitrary fixed
\(u,\,v:\,\alpha<u<v<\beta\). The formula \eqref{MT}, applied to
this \(f\), recovers the starting function \(\varphi\). If a
function \(\varphi(\zeta)\) is defined  only on the straight line
\(\lbrace\zeta:\,\,\zeta=\frac{1}{2}+i\eta,\,-\infty<\eta<\infty\rbrace\)
and satisfy the condition
\(\int\limits_{-\infty}^{\infty}\big|%
\varphi\big(\frac{1}{2}+i\eta\big)\big|^2d\eta<\infty\), we assign
a meaning to the formula \eqref{IFMT} by an approximation
reasoning, which uses the Parseval identity \eqref{PIMT}.

\textit{Thus, there is one-to-one correspondence between functions
\(f(t)\) on the positive half-axis satisfying the condition \(
\int\limits_{0}^{\infty}|f(t)|^2\,dt<\infty\) and the functions
\(\varphi\big(\frac{1}{2}+i\eta\big)\) satisfying the condition
\(\int\limits_{-\infty}^{\infty}\big|%
\varphi\big(\frac{1}{2}+i\eta\big)\big|^2d\eta<\infty\). This
correspondence is established by the pair of formulas \eqref{MT}
and \eqref{IFMT}.The Parseval identity  \eqref{PIMT} holds.
 If the functions \(f(t)\) and
\(\varphi\big(\frac{1}{2}+i\eta\big)\) decay fast enough, then the
integrals in \eqref{MT} and \eqref{IFMT} are understood in a
literal sense, as Lebesgue integrals. For general (not fast
decaying) functions \(f\) and \(\varphi\) from \(L^2\), the
integrals in \eqref{MT} and \eqref{IFMT} are assigned with a
meaning by approximation reasonings.}

Before to prove Theorem \ref{AnDeES}, we curry out some
calculation related to the Gamma function. We use the results of
this calculations in the proof.

\begin{lemma}
The following identities hold:
\begin{subequations}
\begin{gather}
\label{Gam1}%
 \sqrt{\frac{2}{\pi}}\,
\Big(\cos\frac{\pi}{2}\zeta\Big)\,\Gamma(\zeta)=%
2^{\,\zeta-\frac{1}{2}}%
\frac{\Gamma\big(\frac{\zeta}{2}\big)}{\Gamma\big(\frac{1}{2}-\frac{\zeta}{2}\big)}\,,\\
\label{Gam2} \sqrt{\frac{2}{\pi}}\,
\Big(\sin\frac{\pi}{2}\zeta\Big)\,\Gamma(\zeta)=2^{\,\zeta-\frac{1}{2}}%
\frac{\Gamma\big(\frac{1}{2}+\frac{\zeta}{2}\big)}{\Gamma\big(1-\frac{\zeta}{2}\big)}\,.
\end{gather}
\end{subequations}
\end{lemma}
\begin{proof}
From \eqref{Ga2} it follows that
\[\cos\frac{\pi}{2}\zeta=\frac{\pi}{\Gamma\big(\frac{1}{2}-\frac{\zeta}{2}\big)\,%
\Gamma\big(\frac{1}{2}+\frac{\zeta}{2}\big)}\,.\]%
 From \eqref{Ga2} it follows that
 \[\Gamma(\zeta)=
 \pi^{-\frac{1}{2}}\,\Gamma{\textstyle\big(\frac{\zeta}{2}\big)}\,%
 \Gamma{\textstyle\big(\frac{1}{2}+\frac{\zeta}{2}\big)}\,2^{\zeta-1}\,.\]
 Combining the last two formulas, we obtain \eqref{Gam1}.
 Combining the last formula with the formula
 \[\sin\frac{\pi}{2}\zeta=\frac{\pi}{\Gamma\big(\frac{\zeta}{2}\big)\,%
\Gamma\big(1-\frac{\zeta}{2}\big)}\,,\]
 we obtain \eqref{Gam2}.
\end{proof}
\begin{lemma}%
\label{CCI}%
Let \(\zeta\) belongs to the strip \(\Pi\):
\begin{equation}%
\label{VS}%
\Pi=\big\lbrace\zeta:\,0<\textup{Re}\,\zeta<1\,\big\rbrace\,.
\end{equation}%
Then
\begin{subequations}
\label{CcI}
\begin{gather}
\label{CcI1}%
 \int\limits_{0}^{\infty}(\cos{}s)\,s^{\zeta-1}\,ds=
\Big(\cos\,\frac{\pi}{2}\zeta\Big)\,\Gamma(\zeta)\,,\\[1.0ex]
 \label{CcI2}
 \int\limits_{0}^{\infty}(\sin{}s)\,s^{\zeta-1}\,ds=
\Big(\sin\,\frac{\pi}{2}\zeta\Big)\,\Gamma(\zeta)\,,
\end{gather}
\end{subequations}
where the integrals in \eqref{CcI} are understood in the sense
\[\int\limits_{0}^{\infty}
\bigg\lbrace
\begin{matrix}
\cos{}s\\
\sin{}s
\end{matrix}
\bigg\rbrace\, s^{\,\zeta-1}\,ds
=\lim_{R\to+\infty}\int\limits_{0}^{R}\bigg\lbrace
\begin{matrix}
\cos{}s\\
\sin{}s
\end{matrix}
\bigg\rbrace\, s^{\,\zeta-1}\,ds\,,\] and the limit exists
uniformly with respect to \(\zeta\) from any fixed compact subset
of the strip \(\Pi\).
\end{lemma}%
\begin{proof}
We consider the closed contour \(L_R\) which is formed by the
interval \([0,R]\) of the real axis, by the circle \mbox{\(C_R=
\lbrace\zeta:\,\zeta=Re^{i\theta},\,0\leq\theta\leq{}\frac{\pi}{2}\rbrace\)},
and by the interval \([0,iR]\) of the imaginary axes. We choose
the counterclockwise orientation on \(L_R\). We consider the
function \(f(z)=z^{\zeta-1}e^{iz}\) in the domain bounded by
contour \(L_R\). Here \(z^{\zeta-1}=e^{(\zeta-1)\,\ln{}z}\), and
\(0\leq{}\arg{}z\leq{}\frac{\pi}{2}\) in the domain. By Cauchy
theorem, \(\int\limits_{L_R}f(z)\,dz=0\,.\) By Jordan lemma,
\(\lim_{R\to+\infty}\int\limits_{C_R}f(z)\,dz=0\). Analysing the
proof of Jordan lemma, we see, that the last limit is uniform with
respect to \(\zeta\) from any fixed compact subset of the strip
\(\Pi\). The limit \(\lim_{R\to+\infty}\int\limits_{0}^{iR}
f(z)\,dz\) exists end is uniform with respect to \(\zeta\) since
the integral here converges absolutely and uniformly with respect
to \(\zeta\). On the interval \([0,iR],\ \,%
z=se^{i\frac{\pi}{2}}\), where \(s\geq0\), and
\(dz=e^{i\frac{\pi}{2}}\,ds\). Hence
\[\int\limits_{0}^{iR}
f(z)\,dz=\int\limits_{0}^{R}s^{\zeta-1}\,e^{i\frac{\pi}{2}(\zeta-1)}e^{-s}\,
e^{i\frac{\pi}{2}}\,ds=
e^{i\frac{\pi}{2}\zeta}\int\limits_{0}^{R}s^{\zeta-1}e^{-s}\,ds\,.\]
Thus,
\[\lim_{R\to+\infty}\int\limits_{0}^{R}s^{\zeta-1}e^{is}\,ds
=e^{i\frac{\pi}{2}\zeta}\,\Gamma(\zeta)\,.\]
Analogously,
\[\lim_{R\to+\infty}\int\limits_{0}^{R}s^{\zeta-1}e^{-is}\,ds
=e^{-i\frac{\pi}{2}\zeta}\,\Gamma(\zeta)\,.\]
\end{proof}
\begin{lemma}%
\label{MTFT}%
 Let a function \(x(t)\) belongs to \(L^2(0,\infty)\), and \(\hat{x}_c(t)\)
 and \(\hat{x}_s(t)\) are the cosine- and sine- Fourier transform
 of the function \(x\):
 \begin{subequations}
 \label{FuT}
 \begin{gather}
 \label{FuTc}
\hat{x}_c(t)=\sqrt{\frac{2}{\pi}}\int\limits_{0}^{\infty}x(s)\,\cos
(ts)\,ds,\\[1.0ex]
\label{FuTs}
\hat{x}_s(t)=\sqrt{\frac{2}{\pi}}\int\limits_{0}^{\infty}x(s)\,\sin
(ts)\,ds
 \end{gather}
 \end{subequations}
Let \(\varphi_{x}(\zeta),\,\varphi_{\hat{x}_c}(\zeta)\) and
\(\varphi_{\hat{x}_s}(\zeta)\) be the Melline transforms of the
functions \(x,\,\hat{x}_c\) and \(\hat{x}_s\) respectively.
\textup{(}All three functions \(x,\,\hat{x}_c,\,\hat{x}_s\) belong
to \(L^2(0,\infty)\), so their Melline transforms exists and are
\(L^2\) functions on the vertical line
\(\zeta=\frac{1}{2}+i\eta,\,-\infty<\eta<\infty\).\textup{)}

Then for \(\zeta=\frac{1}{2}+i\eta,\,-\infty<\eta<\infty\), the
equalities
\begin{subequations}
 \label{MMT}
\begin{gather}
\label{MMTc}%
 \varphi_{\hat{x}_c}(\zeta)=\varphi_{x}(1-\zeta)\cdot
2^{\,\zeta-\frac{1}{2}}%
\frac{\Gamma\big(\frac{\zeta}{2}\big)}{\Gamma\big(\frac{1}{2}-\frac{\zeta}{2}\big)}\,,\\[1.0ex]
\label{MMTs}
 \varphi_{\hat{x}_s}(\zeta)=\varphi_{x}(1-\zeta)\cdot
2^{\,\zeta-\frac{1}{2}}%
\frac{\Gamma\big(\frac{1}{2}+\frac{\zeta}{2}\big)}{\Gamma\big(1-\frac{\zeta}{2}\big)}\,.
\end{gather}
\end{subequations}
hold.
\end{lemma}
To prove Lemma \ref{MTFT} we need some technical approximation
result.
\begin{definition}
\label{DGF}{\ }\\ %
\hspace*{2.0ex}\textup{1.}
\begin{minipage}[t]{0.94\linewidth}
 A function \(x(t)\) defined on \([0,\infty)\), is said
to be \emph{c-good} if both \(x(t)\) and its cosine Fourier
transform \(\hat{x}_c(t)\) are summable, i.e.
\(\int\limits_{0}^{\infty}|x(t)|\,dt<\infty\),
\(\int\limits_{0}^{\infty}|\hat{x}_c(t)|\,dt<\infty\), and
moreover the functions  \(x(t),\,\hat{x}_c(t)\) are continuous on
\([0,\infty)\) and tend to \(0\) as \(t\to\infty\).
 \end{minipage}{\ }\\ %
\hspace*{2.0ex}\textup{2.}
\begin{minipage}[t]{0.94\linewidth}
 A function \(x(t)\) defined on \([0,\infty)\), is said
to be \emph{s-good} if both \(x(t)\) and its sine Fourier
transform \(\hat{x}_s(t)\) are summable, i.e.
\(\int\limits_{0}^{\infty}|x(t)|\,dt<\infty\),
\(\int\limits_{0}^{\infty}|\hat{x}_s(t)|\,dt<\infty\), and
moreover the functions  \(x(t),\,\hat{x}_s(t)\) are continuous on
\([0,\infty)\) and tend to \(0\) as \(t\to\infty\).
\end{minipage}
\end{definition}
\begin{lemma}
The set of all c-good functions is dense in \(L^2(0,\infty)\).\\
The set of all s-good functions is dense in \(L^2(0,\infty)\).
\end{lemma}
\begin{proof}%
Let \(h_n(t)\) be the functions defined by
\eqref{GrS}-\eqref{LHY}. (Strictly speaking, we consider the
restrictions of the functions  \(h_n(t)\) on \([0,\infty)\).)
 Every function \(h_{2k}\) is c-good: \(\hat{h_{2k}}_c=(-1)^k
 \hat{h_{2k}}\), every function \(h_{2k+1}\) is c-good: \(\hat{h_{2k+1}}_s=(-1)^k
 \hat{h_{2k+1}}\). The set of all finite linear combinations of
 the functions \(h_{2k}\) is dense in \(L^2([0,\infty))\).
  The set of all finite linear combinations of
 the functions \(h_{2k+1}\) also is dense in \(L^2([0,\infty))\).
\end{proof}%
\begin{proof}{Proof of Lemma \ref{MTFT}}.\,%
 We  prove first the equalities \eqref{MMT} assuming that the
 function \(x(t)\) is "good" : c-good if we discuss \eqref{MMTc}
 and s-good if we discuss \eqref{MMTs}. By definition \eqref{MT}
 of the Melline transform,
 \[\varphi_{\hat{x}_c}(\zeta)=%
 \lim_{R\to\infty}\int\limits_{0}^{R}\hat{x}_c(t)t^{\zeta-1}dt\,,\quad
 \textup{Re}\,\zeta=\textstyle\frac{1}{2}\,.\]
Substituting the expression \eqref{FuTc} for \(\hat{x}_c(t)\) into
the last formula, we obtain:
\begin{equation}
\label{T1}%
\varphi_{\hat{x}_c}(\zeta)=%
 \lim_{R\to\infty}\int\limits_{0}^{R}
\bigg(
\sqrt{\frac{2}{\pi}}\int\limits_{0}^{\infty}x(s)\,\cos(ts)\,ds\bigg)
t^{\zeta-1}\,dt\,, \quad
\textup{Re}\,\zeta=\textstyle\frac{1}{2}\,.
\end{equation}
For fixes finite \(R,\,0<R<\infty)\), we change the order of
integration:
\begin{equation}
\label{T2}%
\int\limits_{0}^{R} \bigg(
\sqrt{\frac{2}{\pi}}\int\limits_{0}^{\infty}x(s)\,\cos(ts)\,ds\bigg)
t^{\zeta-1}\,dt=
\int\limits_{0}^{\infty}x(s)\bigg(\sqrt{\frac{2}{\pi}}%
\int\limits_{0}^{R}\cos(ts)\,t^{\zeta-1}\,dt
\bigg)\,ds\,.
\end{equation}%
 The change of order of integration is
justified by Fubini theorem. Changing the variable \(ts=\tau\), we
get
\[\int\limits_{0}^{R}\cos(ts)\,t^{\zeta-1}\,dt=
s^{-\zeta}\int\limits_{0}^{Rs}\cos(\tau)\,\tau^{\zeta-1}\,d\tau\,.\]
Thus%
\begin{multline}%
\label{T3}%
 \int\limits_{0}^{R} \bigg(
\sqrt{\frac{2}{\pi}}\int\limits_{0}^{\infty}x(s)\,\cos(ts)\,ds\bigg)
t^{\zeta-1}\,dt=\\%
=\int\limits_{0}^{\infty}x(s)s^{-\zeta}\bigg(\sqrt{\frac{2}{\pi}}%
\int\limits_{0}^{Rs}\cos(\tau)\,\tau^{\zeta-1}\,d\tau\bigg)\,ds\,.
\end{multline}%
According to  Lemma \ref{CCI}, for every \(s>0\),%
\begin{equation}%
\label{T5}%
\lim_{R\to\infty}\int\limits_{0}^{Rs}\cos(\tau)\,\tau^{\zeta-1}\,d\tau=
\Big(\cos\,\frac{\pi}{2}\zeta\Big)\,\Gamma(\zeta)\,,
\end{equation}%
 For every
\(\zeta: 0<\textup{Re}\zeta<1\),
 the value \(\int\limits_{0}^{\rho}(\cos\tau)\tau^{\zeta-1}\,d\tau\),
 considered as a function of \(\rho\), vanishes at \(\rho=0\),
 is continues function of \(\rho\), and has a finite limit as
 \(\rho\to\infty\),. Therefore there exist a finite
 \(M(\zeta)<\infty\) such that for every
 \(\rho:\,0\leq\rho<\infty\),
 the estimate holds:
 \(\Big|\int\limits_{0}^{\rho}(\cos\tau)\tau^{\zeta-1}\,d\tau\Big|\leq{}M(\zeta)\).
 Here the value \(M(\zeta\) does not depend on \(\rho\). In other
 words,
 \begin{equation*}
 \Bigg|\int\limits_{0}^{Rs}\cos(\tau)\,%
 \tau^{\zeta-1}\,d\tau\Bigg|\leq{}M(\zeta)<
 \infty\quad \forall\, s,R:\,0\leq{}s<\infty,\,0\leq{}R<\infty\,.
 \end{equation*}
 By Lebesgue theorem on dominating convergence,
 \begin{multline}%
 \label{T4}
 \lim_{R\to\infty}\int\limits_{0}^{\infty}x(s)s^{-\zeta}\bigg(\sqrt{\frac{2}{\pi}}%
\int\limits_{0}^{Rs}\cos(\tau)\,\tau^{\zeta-1}\,d\tau\bigg)\,ds=\\
=\int\limits_{0}^{\infty}x(s)s^{-\zeta}\bigg(\sqrt{\frac{2}{\pi}}%
\int\limits_{0}^{\infty}\cos(\tau)\,\tau^{\zeta-1}\,d\tau\bigg)\,ds\,.
\end{multline}%
Unifying \eqref{T1},\,\eqref{T2},\,\eqref{T3},\,\eqref{T4} and
taking into account \eqref{T5} and \eqref{Gam1}, we obtain that
\[\varphi_{\hat{x}_c}(\zeta)=
\int\limits_{0}^{\infty}x(s)\,s^{-\zeta}\,ds\,\cdot\,
2^{\,\zeta-\frac{1}{2}}%
\frac{\Gamma\big(\frac{\zeta}{2}\big)}{\Gamma\big(\frac{1}{2}-\frac{\zeta}{2}\big)}\,.\]
To pass from this formula to \eqref{MMTc}, it remains to observe
that
\[\int\limits_{0}^{\infty}x(s)\,s^{-\zeta}\,ds=\varphi_{x}(1-\zeta)\,.\]
We prove the equality \eqref{MMTc} under assumption that the
function \(x\) is c-good.

To general \(x\in{}L^2([0,\infty))\), the equality \eqref{MMTc}
can be extended by approximation reasoning. Let
\(x\in{}L^2([0,\infty))\), and let \(\big\lbrace{}x_n\big\rbrace\)
be a sequence of c-good functions which converges to \(x\) in
\(L^2([0,\infty))\). The sequence
\(\big\lbrace{}(\widehat{x_n})_c\big\rbrace\) of the
cosine-Fourier transforms will converge to the cosine-Fourier
transform \(\big\lbrace{}\hat{x}_c\big\rbrace\) in
\(L^2([0,\infty))\). Since the correspondence
\(x(t)\to\varphi(\frac{1}{2}+i\eta\)) is an unitary mapping from
\(L^2([0,\infty)\) onto \(L^2(-\infty,\infty)\), see \eqref{IFMT},
then \(\varphi_{\hat{x}_c}\to\varphi_{\hat{x}_c}\), and
\(\varphi_{(\widehat{x_n})_c}\to\varphi_{(\widehat{x_n}})_c\).
(The convergence is in \(L^2\) on the vertical straight line
\(\zeta:\zeta=\frac{1}{2}+i\eta\).) Since the equality
\eqref{MMTc} holds for the functions \(x_n\), and the transform
\(\varphi(\zeta)\to\varphi(1-\zeta)\) is unitary in \(L^2\) on the
vertical line \(\zeta:\zeta=\frac{1}{2}+i\eta\), this equality
can be extended to an arbitrary \(x\in{}L^2([0,\infty))\).

The equality \eqref{MMTs} can be proved analogously.
\end{proof}

\begin{proof}[Proof of \textup{Theorem \ref{AnDeES}}.]{\ } \\[0.5ex]
\hspace*{2.0ex}\textsf{1.}\ Let \(x\) be an eigenfuction of
\(\mathscr{F}\) corresponding to the eigenvalue \(\lambda=1\).
This means that
 \begin{equation}
 \label{CBES1}%
 x(t)=\hat{x}_c(t)\,.
 \end{equation}
 Since the correspondence between functions on \([0,\infty\) and
 their Mellin transforms is one-to-one,
 The equality \eqref{CBES1} is equivalent to the equality
\begin{equation}
 \label{CBES1M}%
\varphi_{x}(\zeta)=\varphi_{\hat{x}_c}(\zeta)\,\quad
\zeta={\textstyle\frac{1}{2}}+i\eta.
 \end{equation}
Taking into account the equality \eqref{MMTc}, we conclude that
the equality \eqref{CBES1} is equivalent to the equalty
\begin{equation}
\label{CBES1Md}%
\varphi_{x}(\zeta)=\varphi_{x}(1-\zeta)\cdot{}
2^{\,\zeta-\frac{1}{2}}%
\frac{\Gamma\big(\frac{\zeta}{2}\big)}{\Gamma\big(\frac{1}{2}-\frac{\zeta}{2}\big)}.
\end{equation}
Substituting \(\zeta=\frac{1}{2}+i\eta\) to the last formula, we
conclude that the equality \eqref{MMTc} is equivalent to the
equality
\begin{equation}%
\label{CBES1Md}%
\varphi_{x}(\tfrac{1}{2}+i\eta)=\varphi_{x}(\tfrac{1}{2}-i\eta)\cdot{}
2^{i\eta}%
\frac{\Gamma\big(\frac{1}{4}+i\frac{\eta}{2}\big)}
{\Gamma\big(\frac{1}{4}-i\frac{\eta}{2}\big)}.
\end{equation}%
The last equality means that the function
\begin{equation}
\label{psee1}%
 \psi_x(i\eta)=\frac{\varphi_{x}(\tfrac{1}{2}+i\eta)}%
 {\Gamma\big(\frac{1}{4}+i\frac{\eta}{2}\big)}\,2^{-i\frac{\eta}{2}}
\end{equation}
is even:
\begin{equation}
\label{Ps1e}
 \psi_x(i\eta)= \psi_x(-i\eta),\quad -\infty<\eta<\infty\,.
\end{equation}
Thus, the condition \eqref{Ps1e}, together with the condition
\[\int\limits_{-\infty}^{\infty}|\psi_{x}(\tfrac{1}{2}+i\eta)\cdot
\Gamma\big(\frac{1}{4}+i\frac{\eta}{2}\big)\,d\eta|^2<\infty\] is
equivalent to the condition  \eqref{MMTc} together with the
condition \(x\in{}L^2([0,\infty)\).

 The statement 1 of Theorem \ref{AnDeES} is proved. The Statements
 2, 3 and 4 are proved in the same way.
\end{proof}


\vspace*{5.0ex}
\noindent
\begin{minipage}[h]{0.45\linewidth}
Victor Katsnelson\\[0.2ex]
Department of Mathematics\\
The Weizmann Institute\\
Rehovot, 76100, Israel\\[0.1ex]
e-mail:\\
{\small\texttt{victor.katsnelson@weizmann.ac.il}}
\end{minipage}
\vspace*{3.0ex}

\noindent
\begin{minipage}[h]{0.45\linewidth}
Ronny Machluf\\[0.2ex]
Department of Mathematics\\
The Weizmann Institute\\
Rehovot, 76100, Israel\\[0.1ex]
e-mail:\\
\texttt{ronny-haim.machluf@weizmann.ac.il}
\end{minipage}
\end{document}